\documentclass[11pt]{amsart}
\textheight
9in
\usepackage{amsmath,amsthm,amsfonts,amssymb,latexsym,epsfig}
\newtheorem{theorem}{Theorem}[section]

\newtheorem{lemma}{Lemma}[section]

\numberwithin{equation}{section}

\theoremstyle{definition}

\theoremstyle{remark}

\begin{document}
\title{Some completely monotonic functions involving the polygamma functions}
\author{Peng Gao}
\address{Division of Mathematical Sciences, School of Physical and Mathematical Sciences,
Nanyang Technological University, 637371 Singapore}
\email{penggao@ntu.edu.sg}
%%\date{\today}
%%\date{January 5, 2010.}

%%\normalsizedbl
%%-------------------------------------------------------------------
\begin{abstract}
%%\normalsizedbl
Motivated by existing results, we present some completely
monotonic functions involving the polygamma functions.
\end{abstract}

\maketitle
%%-----------------------------------------------------------------------
\section{Introduction}
\label{sec 1} \setcounter{equation}{0}
%%------------------------------------------------------------------------
   The digamma (or psi) function  $\psi(x)$ for $x>0$ is defined to be the logarithmic derivative of
    Euler's gamma function
\begin{equation*}
    \Gamma(x)=\int^{\infty}_0 t^x e^{-t}\frac
    {dt}{t}.
\end{equation*}
    The function $\psi$ and its derivatives are called polygamma
    functions.

   There are many interesting inequalities involving the polygamma functions in the literature,
   many of which are closely related to the fact that $\psi'$ is completely
    monotonic on $(0, +\infty)$. Here we recall that a function $f(x)$ is
said to be completely monotonic on $(a, b)$ if it has derivatives
of all orders and $(-1)^kf^{(k)}(x) \geq 0, x \in (a, b), k \geq
0$ and $f(x)$ is said to be strictly completely monotonic on $(a,
b)$ if $(-1)^kf^{(k)}(x) > 0, x \in (a, b), k \geq 0$.

    A general result of Fink \cite[Theorem 1]{F} on completely monotonic functions
    implies that for integers $n \geq 2$,
\begin{align*}
    \left ( \psi^{(n)}(x) \right )^2 \leq
    \psi^{(n-1)}(x)\psi^{(n+1)}(x),  \ x>0.
\end{align*}

     The following inequality of the reverse direction is
     given in \cite{T&W&W}:
\begin{align*}
    \frac {1}{2}\psi^{'}(x)\psi^{'''}(x) \leq \left ( \psi^{''}(x) \right )^2,  \ x>0.
\end{align*}
   A short proof of the above inequality is given in \cite{E&R}.

    For integers $p \geq m \geq n \geq q \geq 0$ and any real number $s$, we
   define
\begin{equation*}
   F_{p, m,n,q}(x; s)= (-1)^{m+n}\psi^{(m)}(x)\psi^{(n)}(x) -
   s(-1)^{p+q}\psi^{(p)}(x)\psi^{(q)}(x),
\end{equation*}
   where we set $\psi^{(0)}(x)=-1$ for convenience.

   In \cite[Theorem 2.1]{A&W}, Alzer and Wells established a nice
    generalization of the above results. Their result  asserts that for $n \geq 2$,
   the function $F_{n+1, n, n, n-1}(x; s)$ is strictly completely monotonic on $(0, +\infty)$ if and only if $s \leq
   (n-1)/n$ and $-F_{n+1, n, n, n-1}(x; s)$ is strictly completely monotonic on $(0, +\infty)$ if and only if $s \geq
   n/(n+1)$.

   We denote
\begin{equation*}
  \alpha_{p,m,n,q} =
     \frac {(m-1)!(n-1)!}{(p-1)!(q-1)!}, \  q \geq 1;
     \ \alpha_{p,m,n,0} =  \frac {(m-1)!(n-1)!}{(p-1)!}; \
\beta_{p,m,n,q} =
     \frac {m!n!}{p!q!}.
\end{equation*}
  Note that $0<\alpha_{p,m,n,q}, \beta_{p,m,n,q}<1$ when $p+q=m+n, p>m$.

   In \cite[Theorem 5.1]{Gao1}, the following generalization of the result of Alzer and Wells
   is given:
\begin{theorem}
\label{thm4}
   Let $p>m \geq n >q \geq 0$ be integers satisfying $m+n=p+q$. The
   function
   $F_{p, m,n,q}(x; \alpha_{p,m,n,q})$ is completely monotonic on $(0, +\infty)$. The function $-F_{p, m,n,q}(x; \beta_{p, m,n, q})$
   is also completely monotonic on $(0, +\infty)$ when $q>0$.
\end{theorem}

    For a given function $f(x)$, we denote for $c>0$,
\begin{align*}
   \Delta f(x;c)=\frac {f(x+c)-f(x)}{c}.
\end{align*}
    We define for integers $p \geq m \geq n \geq q \geq 0$, real number $c>0$ and any real number
    $s$,
\begin{align*}
    F_{p, m,n,q}(x; s;c)= (-1)^{m+n}\Delta\psi^{(m-1)}(x;c)\Delta\psi^{(n-1)}(x;c) -
    s(-1)^{p+q}\Delta\psi^{(p-1)}(x;c)\Delta\psi^{(q-1)}(x;c),
\end{align*}
   where we set $\psi^{(0)}(x)=\psi(x), \psi^{(-1)}(x)=-x$ for
   convenience. We further define $F_{p, m,n,q}(x; s;0)=\lim_{c
   \rightarrow 0^+}F_{p, m,n,q}(x; s;c)$ and it is then easy to
   see that $F_{p, m,n,q}(x; s;0)=F_{p, m,n,q}(x; s)$.

   It's shown in \cite{Q&G} that on $(-\min (s,t), +\infty)$, the function $F_{2,1,1,0}(x+s;
   1;t-s)$ (resp. its negative) is completely monotonic when $|t-s|<1$ (resp. when $|t-s|<1$) and it is
   further given in \cite{Q&G2} a necessary and sufficient condition on $\lambda, t, s$ for $F_{2,1,1,0}(x+s;
   \lambda;t-s)$ or it's negative to be completely monotonic on $(-\min (s,t),
   +\infty)$. We point out here that one can easily deduce these results on $F_{2,1,1,0}(x+s;
   \lambda;t-s)$ from similar results on $F_{2,1,1,0}(x;
   \lambda;t)$ by a change of variable.

   Motivated by the above results, it is our goal in this paper to prove the following:
\begin{theorem}
\label{thm1}
    Let $p>m \geq n >q \geq 0$ be integers satisfying $m+n=p+q$ and let $c>0$. Then
\begin{enumerate}
\item For $0< c \leq 1$,
\begin{enumerate}
\item The function $F_{p, m,n,q}(x; s;c)$ is completely monotonic
on $(0, +\infty)$ if and only if $s \leq \alpha_{p,m,n,q}$.  \item
The function $-F_{m+n, m,n,0}(x; s;c)$ is completely monotonic on
$(0, +\infty)$ if and only if $s \geq \alpha_{m+n,m,n,0}/c$.
\end{enumerate}
\item For $c \geq 1$,
\begin{enumerate}
\item The function $-F_{p, m,n,q}(x; s;c)$ is completely monotonic
on $(0, +\infty)$ if and only if $s \geq \alpha_{p,m,n,q}$.  \item
The function $F_{m+n, m,n,0}(x; s;c)$ is completely monotonic on
$(0, +\infty)$ if and only if $s \leq \alpha_{m+n,m,n,0}/c$.
\end{enumerate}
\item The function $-F_{p, m,n,q}(x; \beta_{p, m,n, q};c)$ is
completely monotonic on $(0, +\infty)$ for all $c>0$ when $q \geq
1$ .
\end{enumerate}

\end{theorem}
%%-----------------------------------------------------------------------
\section{Lemmas}
\label{sec 3} \setcounter{equation}{0}
%%------------------------------------------------------------------------
   The first Lemma lists some facts about the polygamma functions. These can be found, for example, in
  \cite[(1.1)-(1.3), (1.5)]{alz2}:
\begin{lemma}
\label{lem3.0}
  For $x>0$,
\begin{eqnarray}
\label{1.11}
   \psi(x) &=& -\gamma+\int^{\infty}_0\frac {e^{-t}-e^{-xt}}{1-e^{-t}}dt, \ \gamma=0.57721...; \\
 \label{1.10}
  (-1)^{n+1}\psi^{(n)}(x) &=& \int^{\infty}_0e^{-xt}\frac {t^n}{1-e^{-t}}dt
  , \hspace{0.1in} n \geq
  1; \\
\label{3.4}
  \psi^{(n)}(x+1) &=& \psi^{(n)}(x)+(-1)^n\frac {n!}{x^{n+1}},
  \hspace{0.1in} n \geq 0; \\
\label{1.2}
  (-1)^{n+1}\psi^{(n)}(x) &=& \frac
  {(n-1)!}{x^n}+\frac{n!}{2x^{n+1}}+O(\frac{1}{x^{n+2}}), \hspace{0.1in} n \geq 1, \hspace{0.1in} x
\rightarrow +\infty.
\end{eqnarray}
\end{lemma}

\begin{lemma}[{\cite[Lemma 2.9]{Gao1}}]
\label{lem5}
   Let $m>n\geq 1$ be two integers, then for any fixed constant $0<c<1$,
   the function
\begin{equation*}
   a(t;m,n,c)=t^{m-n}+t^n-c(1+t^m)
\end{equation*}
   has exactly one root when $t \geq 1$.
\end{lemma}

\begin{lemma}
\label{lem7}
   Let $a, c>0$, then the function
\begin{equation*}
   u(s;a,c)=\frac {1-e^{-ac(1-s)}}{1-e^{-a(1-s)}} \cdot \frac {1-e^{-ac(1+s)}}{1-e^{-a(1+s)}}
\end{equation*}
    is decreasing on $s \in (0,1)$ if $0< c \leq 1$ and increasing on
    $(0,1)$ if $c \geq 1$.
\end{lemma}
\begin{proof}
    For fixed $c$, it's easy to see that
\begin{align*}
   \frac {u'(s;a,c)}{u(s;a,c)}=a\left (v_c \left (a\left (1-s \right ) \right  )-v_c \left (a\left (1+s \right ) \right  )
   \right),
\end{align*}
   where
\begin{align}
\label{2.3}
   v_c(x)=\frac {1}{e^x-1}-\frac {c}{e^{cx}-1}.
\end{align}
   It's also easy to see that $v'_c(x)=z(x,c)-z(x,1)$ where
\begin{align*}
   z(x,c)=\frac {c^2e^{cx}}{(e^{cx}-1)^2}.
\end{align*}
    Now, we have
\begin{align*}
   \frac {\partial z}{\partial c} =\frac {f(cx)ce^{cx}}{(e^{cx}-1)^3},
\end{align*}
    where $f(t)=(2-t)e^t-(2+t)$. It's then easy to see that $f(t)
    \leq 0$ for $t \geq 0$ and it follows that $v'_c(x) \geq 0$ when
    $0<c \leq 1$ and that $v'_c(x) \leq 0$ when $c \geq 1$.
    We then
    deduce that $u'(s;a,c) \leq 0$ when $0<c\leq 1$ and $u'(s;a,c) \geq 0$ when $c \geq
    1$ and this completes the proof.
\end{proof}
%%-----------------------------------------------------------------------
\section{Proof of Theorem \ref{thm1}}
\label{sec 4.2} \setcounter{equation}{0}
%%------------------------------------------------------------------------

   We first prove assertions (1) (a) and (2) (a) of the theorem.
  Note first that if $F_{p, m,n,q}(x; s;c)$ is completely
  monotonic on $(0, +\infty)$, then we have
\begin{align*}
  s \leq \frac
  {(-1)^{m+n}\Delta \psi^{(m-1)}(x;c)\Delta \psi^{(n-1)}(x;c)}{(-1)^{p+q}\Delta \psi^{(p-1)}(x;c)\Delta \psi^{(q-1)}(x;c)}.
\end{align*}
    It then follows easily from the mean value theorem and \eqref{1.2} that we have
\begin{align*}
   \lim_{x \rightarrow +\infty}\frac
  {(-1)^{m+n}\Delta \psi^{(m-1)}(x;c)\Delta \psi^{(n-1)}(x;c)}{(-1)^{p+q}\Delta \psi^{(p-1)}(x;c)\Delta \psi^{(q-1)}(x;c)}=\alpha_{p,m,n,q}.
\end{align*}
   Thus, $s \leq \alpha_{p,m,n,q}$. Similarly, one shows that if $-F_{p, m,n,q}(x; s;c)$ is completely
  monotonic on $(0, +\infty)$, then $s \geq \alpha_{p,m,n,q}$ and this
  proves the ``only if" part of the assertions (1) (a) and (2) (a) of the theorem.

  To prove the ``if" part of the assertions (1) (a) and (2) (a) of the theorem, it's
  easy to see that it suffices to show that $F_{p, m,n,q}(x; \alpha_{p,m,n,q}; c)$ is
completely monotonic on $(0, +\infty)$ when $0< c \leq 1$ and that
$-F_{p, m,n,q}(x; \alpha_{p,m,n,q}; c)$ is completely
  monotonic on $(0, +\infty)$ when $c \geq 1$.

  We first consider the function $F_{p, m,n,q}(x; \alpha_{p,m,n,q}; c)$ with
  $q \geq 1$ following the approach in \cite{A&W}.
  Using the integral
  representations \eqref{1.11} and \eqref{1.10} for the polygamma functions and using $*$ for the
  Laplace convolution, we get
\begin{equation*}
   F_{p, m,n,q}(x; \alpha_{p,m,n,q};c)=\int^{\infty}_0\frac {e^{-xt}}{c^2}g_{p,m,n,q}(t;\alpha_{p,m,n,q})dt,
\end{equation*}
   where
\begin{eqnarray*}
   g_{p,m,n,q}(t;\alpha_{p,m,n,q}) &=& \frac {t^{m-1}(e^{-ct}-1)}{1-e^{-t}} * \frac {t^{n-1}(e^{-ct}-1)}{1-e^{-t}}-\alpha_{p,m,n,q}\frac
{t^{p-1}(e^{-ct}-1)}{1-e^{-t}} * \frac {t^{q-1}(e^{-ct}-1)}{1-e^{-t}} \\
  &=& \int^t_{0}\Big ( (t-s)^{m-1}s^{n-1}-\alpha_{p,m,n,q}(t-s)^{p-1}s^{q-1} \Big )h_c(t-s)h_c(s)ds,
\end{eqnarray*}
  with
\begin{equation}
\label{3.1}
   h_c(s)=\frac {1-e^{-cs}}{1-e^{-s}}.
\end{equation}
   By a change of variable
   $s \rightarrow ts$ we can recast $g(t)$ as
\begin{equation*}
   g_{p,m,n,q}(t;\alpha_{p,m,n,q})=t^{m+n-1}\int^1_{0}\Big ( (1-s)^{m-1}s^{n-1}-\alpha_{p,m,n,q}(1-s)^{p-1}s^{q-1} \Big )
   h_c(t(1-s))h_c(ts)ds.
\end{equation*}
   We now break the above integral into two integrals, one from $0$
   to $1/2$ and the other from $1/2$ to $1$. We
   make a further change of variable $s \rightarrow (1-s)/2$ for the first one and
   $s \rightarrow (1+s)/2$ for the second one. We now combine them
   to get
\begin{align}
\label{3.02}
  &g_{p,m,n,q}(t;\alpha_{p,m,n,q}) \\
  =&\Big(\frac {t}2 \Big )^{m+n-1}\int^1_{0}a(\frac {1+s}{1-s}; p-q, n-q, \alpha_{p,m,n,q})
  (1-s^2)^{q-1}(1-s)^{p-q}
   u(s;t/2,c)ds,  \nonumber
\end{align}
   where the function $a(t;m,n,c)$ is defined as in Lemma
   \ref{lem5} and the function $u(s;a,c)$ is defined as in Lemma \ref{lem7}. Note that $(1+s)/(1-s) \geq 1$ for $0 \leq s <1$ and $p-q>n-q \geq 1$,
   hence by Lemma \ref{lem5}, there is a unique number $0 < s_0< 1$ such
   that
\begin{equation*}
   a \Big ( \frac {1+s_0}{1-s_0}; p-q, n-q, \alpha_{p,m,n,q} \Big)=0.
\end{equation*}
    It follows from $a(1; p-q, n-q, \alpha_{p,m,n,q})>0$ and $\lim_{t
    \rightarrow +\infty}a(t; p-q, n-q, \alpha_{p,m,n,q})<0$ that for $0< s
    \leq  s_0$,
\begin{align*}
   a \Big ( \frac {1+s}{1-s}; p-q, n-q, \alpha_{p,m,n,q} \Big) \geq 0,
\end{align*}
   with the above inequality being reversed when $s_0 \leq s <1$.

   We further note by Lemma \ref{lem7}, the function $u(s;t/2,c)$
   is decreasing on $s \in (0,1)$ when $0<c \leq 1$ and increasing when $c \geq 1$.
   Thus we conclude that when $0<c \leq 1$,
\begin{eqnarray*}
   && a(\frac {1+s}{1-s}; p-q, n-q, \alpha_{p,m,n,q})
  (1-s^2)^{q-1}(1-s)^{p-q}
   u(s;t/2,c) \\
   &\geq & a(\frac {1+s}{1-s}; p-q, n-q, \alpha_{p,m,n,q})
  (1-s^2)^{q-1}(1-s)^{p-q}
  u(s_0;t/2,c),
\end{eqnarray*}
   with the above inequality being reversed when $c \geq 1$.

   Hence when $0<c \leq 1$,
\begin{eqnarray*}
   &&g_{p,m,n,q}(t;\alpha_{p,m,n,q}) \\
   &\geq& \Big(\frac {t}2 \Big )^{m+n+1}u(s_0;t/2,c) \int^1_{0}a(\frac {1+s}{1-s}; p-q, n-q, \alpha_{p,m,n,q})
  (1-s^2)^{q-1}(1-s)^{p-q}ds,
\end{eqnarray*}
   with the above inequality being reversed when $c \geq 1$.

  Note that the integral above is (by reversing the process above on
  changing variables)
\begin{equation*}
   2^{m+n-1}\int^1_{0}\Big ( (1-s)^{m-1}s^{n-1}-\alpha_{p,m,n,q}(1-s)^{p-1}s^{q-1} \Big
   )ds=0,
\end{equation*}
   where the last step follows from the well-known beta function
   identity
\begin{equation*}
   B(x,y)=\int^1_0t^{x-1}(1-t)^{y-1}dt=\frac
   {\Gamma(x)\Gamma(y)}{\Gamma(x+y)}, \hspace {0.1in} x,y>0,
\end{equation*}
   and the well-known fact $\Gamma(n) = (n-1)!$ for $n \geq 1$. It
   follows that $g(t) \geq 0$ when $0<c \leq 1$ and $g(t) \leq 0$
   when $c \geq 1$ and this completes the proof for the ``if" part of the assertions (1) (a) and (2) (a) of Theorem \ref{thm1}
   for $F_{p, m,n,q}(x; \alpha_{p,m,n,q}; c)$ with $q \geq 1$.

   Now we consider the function $F_{p, m,n,q}(x; \alpha_{p,m,n,q};c)$ with $q=0$.
In this case $p=m+n$
   and we note that
\begin{equation*}
   \alpha_{m+n,m,n, 0} = B(m,n)=\int^1_0s^{m-1}(1-s)^{n-1}ds,
\end{equation*}
   and we use this to write
\begin{equation*}
   \alpha_{m+n,m,n,0}\frac {t^{m+n-1}(e^{-ct}-1)}{1-e^{-t}}=\int^t_0\frac
   {s^{m-1}(t-s)^{n-1}(e^{-ct}-1)}{1-e^{-t}}ds.
\end{equation*}
   It follows that
\begin{eqnarray}
\label{3.2}
   && F_{m+n, m,n,0}(x; \alpha_{m+n,m,n,0};c) \\
   &=& \int^{\infty}_0\frac {e^{-xt}}{c^2} \Big (
   \frac {t^{m-1}(e^{-ct}-1)}{1-e^{-t}} * \frac {t^{n-1}(e^{-ct}-1)}{1-e^{-t}} +\alpha_{m+n,m,n,0}\frac {ct^{m+n-1}(e^{-ct}-1)}{1-e^{-t}} \Big
   )dt \nonumber \\
   &=&  \int^{\infty}_0\frac {e^{-xt}}{c^2} \left ( \int^t_0s^{m-1}(t-s)^{n-1}
   \left ( \frac {1-e^{-cs}}{1-e^{-s}}\cdot\frac {1-e^{-c(t-s)}}{1-e^{-(t-s)}}-
   \frac {c(1-e^{-ct})}{1-e^{-t}}
   \right )ds \right ) dt. \nonumber
\end{eqnarray}
    Now we note that, for $h_c(s)$ defined as in \eqref{3.1},
\begin{align*}
    \frac {h'_c(s)}{h_c(s)}=-v_c(s),
\end{align*}
    where $v_c(x)$ is defined as in \eqref{2.3}. It follows from the
    proof of Lemma \ref{lem7} that $h'_c(s)/h_c(s) \leq 0$ when $0 < c
    \leq 1$ and $h'_c(s)/h_c(s) \geq 0$ when $c \geq 1$. We then deduce
    that when $0<c \leq 1$,
\begin{equation*}
  \frac {h'_c(t-s)}{h_c(t-s)}-\frac
  {h'_c(t)}{h_c(t)} \geq 0,
\end{equation*}
  for $t \geq s \geq 0$ with the above inequality being reversed when $c \geq 1$.
  This implies that the function $t \mapsto \ln h_c(t-s)-\ln h_c(t)$ is increasing (resp. decreasing) for $t >s$ when $0< c \leq 1$
  (resp. when $c \geq 1$). Thus we
  obtain that when $0<c \leq 1$,
\begin{equation*}
   \ln h_c(s)+\ln h_c(t-s) - \ln h_c(t) \geq  \lim_{t \rightarrow s^+}(\ln h_c(s)+\ln h_c(t-s) - \ln h_c(t))
   =\ln c,
\end{equation*}
   with the above inequality being reversed when $c \geq 1$. One checks
   easily that this implies that when $0<c \leq 1$,
\begin{align*}
   \frac {1-e^{-c(t-s)}}{1-e^{-(t-s)}} \geq
   \frac {c(1-e^{-ct})}{1-e^{-t}},
\end{align*}
   with the above inequality being reversed when $c \geq 1$. This
   implies the ``if" part of the assertions (1) (a) and (2) (a) of the theorem
   for $F_{m+n, m,n,0}(x; \alpha_{m+n,m,n,0}; c)$.

   Now we prove the assertions (1) (b) and (2) (b) of the theorem.
  Note that if $F_{m+n, m,n,0}(x; s;c)$ is completely
  monotonic on $(0, +\infty)$, then we have
\begin{align*}
  s \leq \frac
  {(-1)^{m+n}\Delta \psi^{(m-1)}(x;c)\Delta\psi^{(n-1)}(x;c)}{(-1)^{m+n+1}\Delta\psi^{(m+n-1)}(x;c)}.
\end{align*}
    It then follows from \eqref{3.4} that we have
\begin{align*}
   \lim_{x \rightarrow 0^+}\frac
  {(-1)^{m+n}\Delta\psi^{(m-1)}(x;c)\Delta\psi^{(n-1)}(x;c)}{(-1)^{m+n+1}\Delta\psi^{(m+n-1)}(x;c)}=\frac {\alpha_{m+n,m,n,0}}{c}.
\end{align*}
    Thus, $s \leq \alpha_{m+n,m,n,0}/c$. Similarly, one shows that if $-F_{m+n, m,n,0}(x; s;c)$ is completely
  monotonic on $(0, +\infty)$, then $s \geq \alpha_{m+n,m,n,0}/c$ and this
  proves the ``only if" part of the assertions (1) (b) and (2) (b) of the theorem.

   To prove the ``if" part of the assertions (1) (b) and (2) (b) of the theorem, it's
  easy to see that it suffices to show that $-F_{m+n, m,n,0}(x; \alpha_{m+n,m,n,0}/c; c)$ is
completely monotonic on $(0, +\infty)$ when $0< c \leq 1$ and that
$F_{m+n, m,n,0}(x; \alpha_{m+n,m,n,0}/c; c)$ is completely
  monotonic on $(0, +\infty)$ when $c \geq 1$.

  Similarly to \eqref{3.2}, we have
\begin{align*}
   & F_{m+n, m,n,0}(x; \frac {\alpha_{m+n,m,n,0}}{c};c) \\
   =&  \int^{\infty}_0\frac {e^{-xt}}{c^2} \left ( \int^t_0s^{m-1}(t-s)^{n-1}
   \left ( \frac {1-e^{-cs}}{1-e^{-s}}\cdot\frac {1-e^{-c(t-s)}}{1-e^{-(t-s)}}-
   \frac {(1-e^{-ct})}{1-e^{-t}}
   \right )ds \right ) dt.
\end{align*}
   For fixed $t>s>0$, define
\begin{align*}
   r_{s,t}(c)=\frac {(1-e^{-cs})(1-e^{-c(t-s)})}{1-e^{-ct}}.
\end{align*}
    Then we have
\begin{align*}
   c\frac {r'_{s,t}(c)}{r_{s,t}(c)}=\frac {cs}{e^{cs}-1}+\frac
   {c(t-s)}{e^{c(t-s)}-1}-\frac {ct}{e^{ct}-1}>0,
\end{align*}
   as it's easy to see that the function $x \mapsto x/(e^x-1)$ is
   decreasing for $x > 0$. It follows that the function $r_{s,t}(c)$
   is an increasing function of $c$ so that $r_{s,t}(c) \leq r_{s,t}(1)$ when
   $0<c \leq 1$ and $r_{s,t}(c) \geq r_{s,t}(1)$ when
   $c \geq 1$. One sees easily that the ``if" part of the assertions (1) (b) and (2) (b) of the
   theorem follows from this.

   Lastly, we prove assertion (3) of the theorem. This is similar
   to our proof above of the ``if" part of the assertions (1) (a) and (2) (a) of the theorem for $F_{p, m,n,q}(x; \alpha_{p,m,n,q}; c)$ with
  $q \geq 1$, except that we replace $\alpha_{p,m,n,q}$ by $\beta_{p,m,n,q}$ and recast the function $g_{p,m,n,q}(t;\beta_{p,m,n,q})$ similar to
  \eqref{3.02} as
\begin{align*}
  &g_{p,m,n,q}(t;\beta_{p,m,n,q}) \\
  =&\Big(\frac {t}2 \Big )^{m+n+1}\int^1_{0}a(\frac {1+s}{1-s}; p-q, n-q, \beta_{p,m,n,q})
  (1-s^2)^{q}(1-s)^{p-q}\left(\frac {t^2}{4}(1-s^2) \right )^{-1}
   u(s;t/2,c)ds.
\end{align*}
   It is then easy to show using the method in the proof of Lemma \ref{lem7} that the function
   $s \mapsto a^{-2}(1-s^2)^{-1}u(s;a,c)$ is increasing on $s\in (0, 1)$ when $c >0$ and essentially repeating the rest of the proof of
   the ``if" part of the assertions (1) (a) and (2) (a) of the theorem allows us to
   establish assertion (3) of the theorem.
%%-----------------------------------------------------------------------

%%------------------------------------------------------------------------

\end{document}